\documentclass[12pt,a4paper,intlimits,sumlimits]{amsart}

\usepackage{amsmath, amsthm, mathtools}
\usepackage{amsfonts}
\usepackage[alphabetic, nobysame]{amsrefs}

\usepackage{float}
\usepackage{graphicx}
\usepackage{caption}
\usepackage{subcaption}

\usepackage{framed}
\usepackage{amssymb}
\usepackage{esvect}
\usepackage{xcolor}
\usepackage{eucal}
\usepackage{mathrsfs}
\usepackage{hyperref}
\usepackage[english]{babel}
\usepackage{mathrsfs}
\usepackage{esint}

\def \N {\mathbb{N}}
\def \R {\mathbb{R}}

\theoremstyle{definition}
\newtheorem{definition}{Definition}[section]

\theoremstyle{plain}
\newtheorem{theorem}[definition]{Theorem}
\newtheorem{proposition}[definition]{Proposition}
\newtheorem{lemma}[definition]{Lemma}

\numberwithin{equation}{section}

\usepackage{geometry}
\renewcommand{\epsilon}{\varepsilon}

\renewcommand{\leq}{\leqslant}
\renewcommand{\le}{\leqslant}
\renewcommand{\geq}{\geqslant}
\renewcommand{\ge}{\geqslant}

\allowdisplaybreaks

 \title[Maximum principles and spectral analysis]{Maximum principles and spectral analysis for the superposition of operators of fractional order}

\author[S. Dipierro, E. Proietti Lippi, C. Sportelli and E. Valdinoci]{Serena Dipierro, Edoardo Proietti Lippi, Caterina Sportelli and Enrico Valdinoci}

\address{Serena Dipierro: Department of Mathematics and Statistics, The University of Western Australia, 35 Stirling Highway, Crawley, Perth, WA 6009, Australia}
\email{serena.dipierro@uwa.edu.au}

\address{Edoardo Proietti Lippi: Department of Mathematics and Statistics, The University of Western Australia, 35 Stirling Highway, Crawley, Perth, WA 6009, Australia}
\email{edoardo.proiettilippi@uwa.edu.au}

\address{Caterina Sportelli: Departamento de Análisis Matemático, Universidad de Granada, 18071 Granada, Spain \\\&\\ Department of Mathematics and Statistics, The University of Western Australia, 35 Stirling Highway, Crawley, Perth, WA 6009, Australia}
\email{caterina.sportelli@uwa.edu.au, caterina.sp@ugr.es}

\address{Enrico Valdinoci: Department of Mathematics and Statistics, The University of Western Australia, 35 Stirling Highway, Crawley, Perth, WA 6009, Australia}
\email{enrico.valdinoci@uwa.edu.au}

\begin{document}

\maketitle

\begin{abstract}
We consider a ``superposition operator" obtained through the continuous superposition of operators of mixed fractional order,  modulated by a signed finite Borel measure defined over the set $[0, 1]$.  The relevance of this operator is rooted in the fact that it incorporates special and significant cases of interest, like the mixed operator $-\Delta + (-\Delta)^s$, the (possibly) infinite sum of fractional Laplacians and allows to consider operators carrying a ``wrong sign".

We first outline weak and strong maximum principles for this type of operators. Then, we complete the spectral analysis for the related Dirichlet eigenvalue problem started in~\cite{DPLSVlog}.
\end{abstract}

\begin{figure}[H]
\centering\captionsetup{labelformat=empty}
\includegraphics[width=0.75\textwidth]{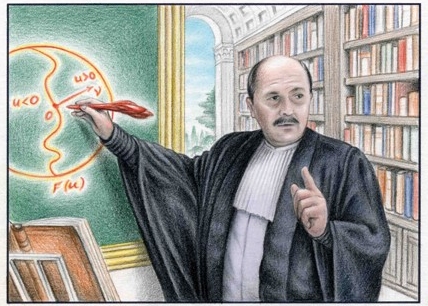}
\caption{Portrait of Sandro by Luigi Serafini.}
\end{figure}

\tableofcontents

\section{Introduction}
\subsection{Why superposition operators?}
Superposition operators naturally arise in the study of anomalous diffusion, with a motivating example coming from population dynamics. Suppose that
a population consists of different groups of individuals, where each group follows a distinct diffusion pattern---such as a Gaussian process or a Lévy flight with varying diffusive exponents. The overall displacement of the population is then governed by an evolution equation in which the diffusion operator is a linear combination of the various operators describing each group's diffusion pattern (see~\cite{MR4249816}).

Superposition operators, where the coefficients of these linear combinations are allowed to change sign, are also useful in describing population density patterns. Positive coefficients correspond to diffusive phenomena, while negative contributions model ``reverse parabolic'' features, which help describe population concentration (see~\cite{ZZLIB, DPLSVlog}.

In this paper we carry on the study of operators derived from the (both finite and infinite,
both discrete and continuous) superposition of operators of mixed fractional order, focusing on two key topics: the validity of some maximum principles and the analysis of the spectral properties of this kind of operators.

\subsection{The mathematical setting}
The superposition of operators that we consider takes the form
\begin{equation}\label{superpositionoperator}
L_\mu u:=\int_{[0, 1]} (-\Delta)^s u\, d\mu(s),
\end{equation}
where $\mu$ denotes the signed measure of the form
\begin{equation}\label{misure+-}
\mu=\mu^+ -\mu^-.
\end{equation}
Here above, $\mu^+$ and~$\mu^-$ are assumed to be two nonnegative finite Borel measures over the interval $[0, 1]$. 
Also, in~\eqref{superpositionoperator}, the notation~$(-\Delta)^s$ stands for the fractional Laplacian, defined, for all~$s\in(0,1)$, as
\begin{equation}\label{AMMU0} (- \Delta)^s\, u(x) = c_{N,s}\int_{\R^N} \frac{2u(x) - u(x+y)-u(x-y)}{|y|^{N+2s}}\, dy.
\end{equation}
The positive normalizing constant~$c_{N,s}$, defined by
\begin{equation}\label{defcNs}
c_{N,s}:=\frac{2^{2s-1}\,\Gamma\left(\frac{N+2s}{2}\right)}{\pi^{N/2}\,\Gamma(2-s)}\, s(1-s),
\end{equation}
being~$\Gamma$ the Gamma function,
is chosen in such a way that, for~$u$ smooth and rapidly decaying, the Fourier transform of~$(-\Delta)^s u$ returns~$(2\pi|\xi|)^{2s}$ times the Fourier transform of~$u$ and provides consistent limits as~$s\nearrow1$ and as~$s\searrow0$, namely
\[
\lim_{s\nearrow1}(-\Delta)^su=(-\Delta)^1u=-\Delta u
\qquad{\mbox{and}}\qquad
\lim_{s\searrow0}(-\Delta)^s u=(-\Delta)^0u=u.
\]
We refer the interested reader to~\cite{MR2944369, 2024arXiv241118238A} for a detailed presentation of the fractional Laplace operator.

The setting in which we operate is the same as the one originally introduced in~\cite{MR4736013}.  In~\cite{DPSV2}, that framework has been further broadened to the superposition of fractional~$p$-Laplacians (see also~\cite{MR4793906}). In addition, the case of fractional $p$-Laplacians of different orders, namely for~$s\in[0,1]$ and~$p\in (1,N)$, has been considered in~\cite{TUTTI3}, whereas the superposition of fractional operators under Neumann boundary conditions has been examined in~\cite{TUTTI, TUTTI2}. 

Now we introduce our setting with all the details. We let~$\mu$ be as in~\eqref{misure+-}.
The components of the measure tied to higher fractional exponents are supposed to be positive, whereas the negative components (if any) must be appropriately reabsorbed into the positive ones.
This is formally expressed by requiring the existence of~$\overline s \in (0, 1]$ and~$\gamma\ge 0$ (to be chosen small enough) such that the following properties are satisfied:
\begin{align}\label{mu0}
&\mu^+\big([\overline s, 1]\big)>0,\\ \label{mu1}
&\mu^-\big([\overline s, 1]\big)=0,\\ \label{mu2}
\mbox{ and } \quad&\mu^-\big([0, \overline s)\big)\le\gamma\mu^+\big([\overline s, 1]\big).
\end{align}

%By assumption~\eqref{mu0},  we infer that
%\begin{equation}\label{scritico}
%\mbox{ there exists $s_\sharp \in [\overline s, 1]$ such that } \mu^+ ([s_\sharp, 1])> 0.
%\end{equation}
The exponent $\overline s$ will serve as a Sobolev critical exponent, its role being enhanced in the upcoming Proposition~\ref{embeddingsXOmega}.
%Roughly speaking, while some arbitrariness is allowed in the choice of~$s_\sharp$ here above,
%the results obtained will be stronger if one takes~$s_\sharp$ ``as large as possible'', but still satisfying condition~\eqref{scritico}.

We now delve into the two key topics of the paper: the validity of a maximum principle and the analysis of the spectral properties of the operator~\eqref{superpositionoperator}.

\subsection{Maximum principles}
We point out that, if the measure in~\eqref{misure+-} comes with no negative contributions, that is~$\mu^- \equiv 0$, then the operator in~\eqref{superpositionoperator} boils down to
\begin{equation}\label{superpositionoperator+}
L_\mu^+ u:=\int_{[0, 1]} (-\Delta)^s u\, d\mu^+(s).
\end{equation}
Our first set of results deals with maximum principles for the operator~$L^+_\mu$ in~\eqref{superpositionoperator+}. We present different formulations, depending on whether weak or classical solutions are being sought.

To this aim, we take~$N>2 \overline s$ and we introduce the space~$X(\R^N)$,
which is given by the completion of compactly supported smooth functions
with respect to a specific norm which takes into account the presence of the measure~$\mu^+$ (see the forthcoming definition in~\eqref{normadefinizione}).
The weak maximum principles for the operator~$L^+_\mu$ reads as follows:

\begin{theorem}\label{weakmaxprinciple}
Let~$\Omega$ be an open and bounded subset of~$\R^N$
with Lipschitz boundary. Let~$\mu$ satisfy~\eqref{mu0}.

Let~$u\in X(\R^N)$ be such that~$L^+_\mu u\geq 0 $ in~$\Omega$
in the weak sense. Suppose\footnote{When $\mu^+$ is the Dirac delta\label{FOO:ABNO5}
at $1$ (or, more generally, if $\mu^+([0,1))=0$) the space $X^+(\Omega)$ coincides with $H^1(\Omega)$, with $u \ge 0$ on $\partial\Omega$ understood in the classical trace sense.} that~$u\ge 0$ a.e. in~$\R^N\setminus \Omega$.

Then, $u\geq 0$ a.e. in~$\Omega$.
\end{theorem}

For classical solutions, the previous statement is enhanced in this way:

\begin{theorem}\label{strongmaximumprinciple}
Let~$\Omega$ be an open and bounded subset of~$\R^N$
with Lipschitz boundary. Let~$\mu$ satisfy~\eqref{mu0}.

Let~$u\in C(\R^N)\cap C^2(\Omega)$.

If~$\mu^+((0,1))>0$, assume also that
\[
\int_{[0,1]}c_{N,s} \int_{\R^N}\frac{|u(x)|}{1+|x|^{N+2s}}\,dx\,d\mu^+(s)<+\infty.
\]
Suppose that 
\[
\begin{cases}
L^+_\mu u \ge 0 &\mbox{in } \Omega,\\
u\ge 0 &\mbox{in } \R^N\setminus \Omega.
\end{cases}
\]

Then,
\begin{equation}\label{2uge0Omega}
u\ge 0\quad\mbox{ in } \Omega.
\end{equation}

In addition, if~$u(x_0)=0$ for some~$x_0\in \Omega$, then~$u\equiv 0$ in the connected component
of~$\Omega $ containing~$x_0$, and in fact in the whole of~$\R^N$
provided that~$\mu^+((0,1))>0$.
\end{theorem}

We stress that Theorem~\ref{weakmaxprinciple} and the statement in~\eqref{2uge0Omega} differ in the assumptions required on the function~$u$, which is assumed to belong to~$X(\R^N)$ in the first case, while a ``classical" setting is required in the statement of Theorem~\ref{strongmaximumprinciple}.
%%We will provide a self-contained proof
%%of these maximum principles in Section~\ref{sectionmaximumprinciples}.

\subsection{Failure of the maximum principle}
We point out that the validity of Theorems~\ref{weakmaxprinciple} and~\ref{strongmaximumprinciple} is heavily impacted by the choice of the leading operator. One might be tempted to consider the more general operator~$L_\mu$ defined in~\eqref{superpositionoperator} in place of~$L^+_\mu$ in the statements of Theorems~\ref{weakmaxprinciple} and~\ref{strongmaximumprinciple}.
Unfortunately, this is not possible. 
Indeed, a classical result by Bony, Courr\`ege and Priouret
(see~\cite[Theorem~3, page~391]{MR245085}) states that for linear translation operators the maximum principle holds if and only if the measure~$\mu$ in the operator in~\eqref{superpositionoperator} has constant sign, that is it holds by taking either~$\mu^+\equiv 0$ or~$\mu^-\equiv 0$ in~\eqref{misure+-}.

For the sake of completeness, we exhibit an explicit counterexample which disproves the validity of the maximum principle in the presence of a negative contribution in the measure~$\mu$.
More precisely, we will show that, for any~$s\in(0,1)$ and~$\alpha\in(0,+\infty)$,
\begin{equation}\label{counterexamplemaxprinciple}
{\mbox{the maximum principle does not hold for the operator~$-\Delta - \alpha(-\Delta)^s$.}}
\end{equation}
To this aim, we will take~$\Omega:=(-1,1)$ and construct a function~$u:\R\to \R$ such that~$u\in C(\R)\cap C^2(\Omega)$ and
\[
-\Delta u - \alpha(-\Delta)^s u \geq 0 \quad\mbox{in } \Omega\quad \mbox{and}\quad u\ge 0 \quad \mbox{in } \R\setminus \Omega,
\]
but~$u(x)<0$ for any~$x\in \Omega$.

Notice that the operator introduced in~\eqref{counterexamplemaxprinciple} corresponds to the choice~$\mu:=\delta_1 -\alpha\delta_s$, being~$\delta_1$ and~$\delta_s$ the Dirac measures centered at the points~$1$ and~$s\in (0, 1)$, respectively,
thus the operator~$-\Delta - \alpha(-\Delta)^s$ occurs as a particular case of~$L_\mu$.

We also point out that the assumptions in~\eqref{mu0} and~\eqref{mu1}
for~$-\Delta - \alpha(-\Delta)^s$ are satisfied, and, by taking~$\alpha$ sufficiently small, the assumption in~\eqref{mu2}
can be fulfilled as well. This says that not even a ``reabsorbing'' property
can make the maximum principle hold true.

The proof of~\eqref{counterexamplemaxprinciple} will be presented in Appendix~\ref{appendixmaxprinciple}.
Also, we refer the interested reader to~\cite[Appendix~A]{MR4387204} for a counterexample in which the ``wrong'' sign occurs in the Laplacian and the maximum principle is proved to fail for
the operator~$\Delta+(-\Delta)^s$
for any~$s\in (0, 1/2)$.

\subsection{Eigenvalues and eigenfunctions}
The other topic that we tackle in this paper is the Dirichlet eigenvalue problem driven by the operator $L_\mu$.  
%The study of the eigenvalues and eigenvectors of operators plays a fundamental role in various areas of both mathematics and physics, offering profound insights into the behavior of dynamical systems. Indeed, the analysis of eigenvalues is crucial for grasping the spectrum of linear operators, and their study provides a tool to analyze the stability,  the resonance, and the long-term behavior of systems modeled by partial differential equations.
To this aim, we let~$\lambda\in\R$ and we take into account the problem
\begin{equation}\label{eigenvalueproblem}
\begin{cases}
L_\mu u= \lambda u &\mbox{ in } \Omega,\\
u=0 &\mbox{ in } \R^N\setminus\Omega.
\end{cases}
\end{equation}

The existence of the first positive eigenvalue of problem~\eqref{eigenvalueproblem}, as well as its properties, have been established in~\cite[Theorem~1.1]{DPLSVlog}.
Here, we complete the spectral analysis of the operator $L_\mu$
by constructing a diverging sequence of eigenvalues~$\lambda_k$ for the problem~\eqref{eigenvalueproblem} such that each eigenvalue has finite multiplicity. 
In addition, the sequence of eigenfunctions~$e_k$ corresponding to~$\lambda_k$ establishes an orthonormal basis for~$L^2(\Omega)$ and an orthogonal basis for the space~$X(\Omega)$, which is defined as the set of functions which belong to~$X(\R^N)$ and vanish outside~$\Omega$.

To state our result, we denote by~$[u]_s$ the Gagliardo semi-norm of~$u$, with the
slight abuse of notation that~$[u]_0=\|u\|_{L^2(\Omega)}$ and~$[u]_1=\|\nabla u\|_{L^2(\Omega)}$ (see
the forthcoming formula~\eqref{seminormgagliardo}).
Moreover, we point out that we write
\begin{equation}\label{ocmrgexpresfyrj5784i}
\int_{[0,1]}c_{N,s}\iint_{\R^{2N}}\frac{(u(x)-u(y))(v(x)-v(y))}{|x-y|^{N+2s}}\,dx\,dy\,d\mu^+(s)\end{equation}
with an abuse of notation. Indeed, to be precise, one should write \label{footimpr2}
\begin{eqnarray*}&&
\int_{(0,1)}c_{N,s}\iint_{\R^{2N}}\frac{(u(x)-u(y))(v(x)-v(y))}{|x-y|^{N+2s}}\,dx\,dy\,d\mu^+(s)
\\&&\qquad\qquad+\mu^+(\{0\})\int_{\Omega}u(x)v(x)\,dx
+\mu^+(\{1\})\int_{\Omega}\nabla u(x)\cdot \nabla v(x)\,dx.
\end{eqnarray*}
We have the following result:

\begin{theorem}\label{thmautovalori}
Let $\mu$ satisfy~\eqref{mu0}, \eqref{mu1} and~\eqref{mu2} and take~$N>2\overline s$.
%Let $s_\sharp$ be as in~\eqref{scritico} and~$N>2s_\sharp$.
Let~$\Omega$ be an open and bounded subset of~$\R^N$ with Lipschitz boundary.

Then, there exists $\gamma_0>0$, depending only on $N$ and $\Omega$, such that if $\gamma\in[0,\gamma_0]$ the following statements hold true.

There exists a sequence~$\lambda_k$ of eigenvalues of problem~\eqref{eigenvalueproblem} such that~$0<\lambda_1 \le \lambda_2\le \cdots\le \lambda_k\le\cdots$, and
\[
\lim_{k\to +\infty} \lambda_k = +\infty.
\]

Moreover,
for any~$k\in \N$, the eigenvalues~$\lambda_{k+1}$
are given recursively by
\begin{equation}\label{kautovaloremisto}
\lambda_{k+1}=\min_{u\in E_{k+1}\setminus \{0\}} \dfrac{\displaystyle\int_{[0,1]}[u]_s^2\,d\mu^+(s)-\int_{[0, \overline s)} [u]^2_s d\mu^-(s)}{\|u\|^2_{L^2(\Omega)}},
\end{equation}
where
where~$E_1:= X(\Omega)$ and, for all~$k\ge1$,
\begin{eqnarray*}
E_{k+1}&:=&\Bigg\{ u\in X(\Omega) \;{\mbox{ s.t. for all }} j=1,\dots,k
\\&&\qquad \int_{[0,1]} c_{N, s}\iint_{\R^{2N}}\dfrac{(u(x)-u(y))(e_j(x)-e_j(y))}{|x-y|^{N+2s}} \,dx\, dy\, d\mu(s)=0
\Bigg\}.
\end{eqnarray*}

Furthermore, for any~$k\in\N$ the function~$e_{k+1}\in E_{k+1}$ is an eigenfunction corresponding to the eigenvalue~$\lambda_{k+1}$ and achieves the minimum in~\eqref{kautovaloremisto}.

Also, the sequence of eigenfunctions~$e_k$ constitutes an orthonormal basis of~$L^2(\Omega)$ and an orthogonal basis of~$X(\Omega)$.

Each eigenvalue~$\lambda_k$ has finite multiplicity, namely, if there exists~$h\in \N$ such that
\[
\lambda_{k-1}<\lambda_k=\cdots =\lambda_{k+h}<\lambda_{k+h+1},
\]
then each eigenfunction corresponding to~$\lambda_k$ belongs to~$\mbox{span}\{e_k,\dots,e_{k+h}\}$.
\end{theorem}

\subsection{Organization of the paper}
The paper is organized as follows. In Section~\ref{sectionpreliminaries} we present the functional framework needed in our setting and we provide some preliminary results.

Section~\ref{sectionmaximumprinciples} deals with the maximum principles and contains the proofs of Theorems~\ref{weakmaxprinciple} and~\ref{strongmaximumprinciple}.
Section~\ref{sectioneigenvalues} contains the study of the eigenvalue problem~\eqref{eigenvalueproblem}.

Moreover, Appendix~\ref{appendixmaxprinciple} contains the proof of the statement in~\eqref{counterexamplemaxprinciple}, while Appendix~\ref{appendixFourier} includes the Fourier analysis needed to show the existence of a basis for the space~$X(\Omega)$ consisting of eigenfunctions.

\section{Preliminary results}\label{sectionpreliminaries}

In this section we introduce the functional analytical setting needed to address the problems introduced in the previous section
and we gather some preliminary observations.

Throughout the paper we assume that~$\Omega$ is an open and bounded subset of~$\R^N$ with Lipschitz boundary.

To this end, for~$s\in[0,1]$, we let
\begin{equation}\label{seminormgagliardo}
[u]_s:=
\begin{cases}
\|u\|_{L^2(\R^N)}  &\mbox{ if } s=0,
\\ \\
\displaystyle\left(c_{N,s}\iint_{\R^{2N}}\frac{|u(x)-u(y)|^2}{|x-y|^{N+2s}}\,dx\,dy \right)^{1/2} &\mbox{ if } s\in(0,1),
\\ \\
\|\nabla u\|_{L^2(\R^N)}  &\mbox{ if } s=1.
\end{cases}
\end{equation}
Here above~$c_{N,s}$ is the constant introduced in~\eqref{defcNs}.

We define the space~$X(\R^N)$ as the completion of~$C^\infty_c (\R^N)$ with respect to the seminorm
\begin{equation}\label{normadefinizione}
\|u\|_{X}:=\left(\;\int_{[0,1]}[u]^2_s\,d\mu^+(s)\right)^{\frac12}.
\end{equation}
Also, for any open and bounded set~$\Omega\subset\R^N$
with Lipschitz boundary, we define the space 
\begin{equation}\label{XOmegadefn}
X(\Omega):=\big\{u\in X(\R^N)\;\mbox{ s.t. }\;
u\equiv 0 \mbox{ in }\R^N\setminus \Omega \big\}.
\end{equation}
We remark that $\|u\|_X$ is a norm for the space $X(\Omega)$ and 
\begin{equation}\label{XHilbert}
\mbox{$X(\Omega)$ is a Hilbert space with respect to~$\|u\|_{X}$.}
\end{equation}
Its associated scalar product is defined, for any~$u$, $v\in X(\Omega)$, as
\begin{equation}\label{scalarepiu}
\langle u, v\rangle_+ :=  \int_{[0,1]}c_{N,s}\iint_{\R^{2N}}\frac{(u(x)-u(y))(v(x)-v(y))}{|x-y|^{N+2s}}\,dx\,dy\,d\mu^+(s).
\end{equation}

%We point out that we write
%\begin{equation}\label{ocmrgexpresfyrj5784i}
%\int_{[0,1]}c_{N,s}\iint_{\R^{2N}}\frac{(u(x)-u(y))(v(x)-v(y))}{|x-y|^{N+2s}}\,dx\,dy\,d\mu^+(s)\end{equation}
%with an abuse of notation. Indeed, to be precise, one should write \label{footimpr2}
%\begin{eqnarray*}&&
%\int_{(0,1)}c_{N,s}\iint_{\R^{2N}}\frac{(u(x)-u(y))(v(x)-v(y))}{|x-y|^{N+2s}}\,dx\,dy\,d\mu^+(s)
%\\&&\qquad\qquad+\mu^+(\{0\})\int_{\Omega}u(x)v(x)\,dx
%+\mu^+(\{1\})\int_{\Omega}\nabla u(x)\cdot \nabla v(x)\,dx.
%\end{eqnarray*}
To ease notation, unless otherwise specified,
we will always use the compact expression introduced in~\eqref{ocmrgexpresfyrj5784i}.

Our study will leverage some Sobolev type and convergence results which we list here. To start with, we observe that higher exponents in fractional norms control the lower exponents through a constant which does not depend on any of the fractional exponents. This fact is a consequence
of the following observation, which is proved in~\cite[Lemma~2.1]{MR4736013}:

\begin{lemma}\label{nons} 
Let~$0\le s_1 \le s_2 \le1$. 

Then, for any measurable function~$u:\R^N\to\R$ with~$u=0$ a.e. in~$\R^N\setminus\Omega$ we have that
\begin{equation}\label{spp}
[u]_{s_1}\le C \, [u]_{s_2},
\end{equation}
for a suitable positive constant~$C=C(N,\Omega)$.
\end{lemma}

It is worth noting that, in our framework, the negative component of the signed measure~$\mu$ can be reabsorbed into the positive one,
provided that~$\gamma$ in~\eqref{mu2} is small enough.
This has been proved in~\cite[Proposition~2.3]{MR4736013}:
 
\begin{proposition}\label{crucial} 
Let~$\mu$ satisfy~\eqref{mu0}, \eqref{mu1} and~\eqref{mu2}.

Then, there exists~$c_0=c_0(N,\Omega)>0$ such that, for any~$u\in X(\Omega)$,
\[
\int_{[{{ 0 }}, \overline s)} [u]_{s}^2 \, d\mu^- (s) \le c_0\,\gamma \int_{[\overline s, 1]} [u]^2_{s} \, d\mu(s).
\]
\end{proposition}

%%%We choose $\gamma$ in~\eqref{mu2} to be small enough, ensuring that
%%%\begin{equation}\label{gammapiccolo}
%%%\gamma<\frac{1}{c_0},
%%%\end{equation}
%%%where $c_0>0$ is taken as in Proposition~\ref{crucial}.

Now, we recall the result stated in~\cite[Proposition~2.4]{MR4736013}, which provides some Sobolev type embeddings for the space~$X(\Omega)$ defined in~\eqref{XOmegadefn}.

\begin{proposition}\label{embeddingsXOmega}
Let~$\mu$ satisfy~\eqref{mu0}, \eqref{mu1} and~\eqref{mu2} and take $N>2\overline s$.
%Let~$s_\sharp$ be as in~\eqref{scritico} and~$N>2s_\sharp$. 

Then, the space~$X(\Omega)$ is continuously embedded in~$H^{\overline s}(\Omega)$.

Furthermore, $X(\Omega)$ is continuously and compactly embedded in~$L^2(\Omega)$. 
%for any~$r\in [1,2^*_{s_\sharp}]$ and compactly embedded in~$L^r(\Omega)$ for any~$r\in [1,2^*_{s_\sharp})$.
\end{proposition}

We notice that, if a sequence is bounded in~$X(\Omega)$, then it converges strongly with respect to the negative part of 
the measure~$\mu$, as shown in~\cite[Lemma~2.8]{DPLSVlog}:

\begin{lemma}\label{lemmalimitemumeno}
Let~$\mu$ satisfy~\eqref{mu0} and~\eqref{mu1}. Let~$u_n$ be a sequence in~$X(\Omega)$ such that~$u_n$ converges weakly to some~$u$ in~$X(\Omega)$
as~$n\to+\infty$.

Then,
\begin{equation}\label{limitemumeno1}
\lim_{n\to +\infty}  \int_{[0,\overline s)} [u_n]^2_s \, d\mu^-(s) =  \int_{[0,\overline s)} [u]^2_s \, d\mu^-(s).
\end{equation}
\end{lemma}

We now introduce a norm which we prove to be equivalent to the one defined in~\eqref{normadefinizione}. This is relevant for our upcoming study of eigenvalues.

To this end, we recall hypothesis~\eqref{mu1} and we also define the inner product
\begin{equation}\label{scalaremeno}
\langle u, v\rangle_- :=  \int_{[0,\overline s)}c_{N,s}\iint_{\R^{2N}}\frac{(u(x)-u(y))(v(x)-v(y))}{|x-y|^{N+2s}}\,dx\,dy\,d\mu^-(s).
\end{equation}
We let
\begin{equation}\label{scalare+-}
\langle u, v\rangle:= \langle u, v\rangle_+ - \langle u, v\rangle_-
\end{equation}
and \begin{equation}\label{norma+-}
\|u\|:= \sqrt{\langle u, u\rangle} = \left(\, \int_{[0,1]}[u]^2_s\,d\mu^+(s) - \int_{[0,\overline s)}[u]^2_s\,d\mu^-(s)\right)^{\frac12},
\end{equation}
and we observe the following:

\begin{proposition}\label{propnormeequiv}
Let $\mu$ satisfy \eqref{mu0}, \eqref{mu1} and~\eqref{mu2}. 

Then, there exists~$\gamma_0>0$, depending only on~$N$ and~$\Omega$, such that, if~$\gamma\in[0,\gamma_0]$,
\eqref{scalare+-} is a dot product in $X(\Omega)$ and
\[
(1-c_0\gamma)^{\frac12}\|u\|_X\le \|u\|\le \|u\|_X,
\]
being $c_0>0$ as in Proposition~\ref{crucial}.
\end{proposition}

\begin{proof}
In light of~\eqref{scalarepiu} and~\eqref{scalaremeno}, we have that~\eqref{scalare+-} is commutative,  distributive and is compatible with the scalar multiplication. 

Moreover, thanks to~\eqref{scalare+-} and Proposition~\ref{crucial},
\[
\begin{split}&
\langle u, u\rangle= \langle u, u\rangle_+ - \langle u, u\rangle_- = \int_{[0,1]}[u]^2_s\,d\mu^+(s) - \int_{[0,\overline s]}[u]^2_s\,d\mu^-(s)\\
&\qquad \ge (1-c_0\gamma)  \int_{[0,1]}[u]^2_s\,d\mu^+(s)\ge 0,
\end{split}
\] provided that~$\gamma$ is sufficiently small. 
Namely, in this case we see that~\eqref{scalare+-} is positive definite, and therefore it identifies a dot product.

Now, by~\eqref{norma+-} and~\eqref{normadefinizione}, we have
\[
\|u\|= \left(\, \int_{[0,1]}[u]^2_s\,d\mu^+(s) - \int_{[0,\overline s]}[u]^2_s\,d\mu^-(s)\right)^{\frac12}\le \left(\, \int_{[0,1]}[u]^2_s\,d\mu^+(s)\right)^{\frac12} = \|u\|_X.
\]
Moreover, by Proposition~\ref{crucial}, we deduce that
\begin{eqnarray*}&&
\|u\|= \left(\, \int_{[0,1]}[u]^2_s\,d\mu^+(s) - \int_{[0,\overline s]}[u]^2_s\,d\mu^-(s)\right)^{\frac12}\\&&\qquad\
\ge \left((1-c_0\gamma)  \int_{[0,1]}[u]^2_s\,d\mu^+(s)\right)^{\frac12} = (1-c_0\gamma)^{\frac12} \|u\|_X.
\end{eqnarray*}
Combining the last two displays, we obtain the desired result.
\end{proof}

\section{Maximum principles}\label{sectionmaximumprinciples}

In this section we provide the proofs of some maximum principles for the operator~$L^+_\mu$. 

Our first main result establishes the maximum principle for weak solutions, in the sense of the following definition (recall that the ``abuse of notation'' in~\eqref{ocmrgexpresfyrj5784i} is assumed):

\begin{definition}
Let~$f\in L^2(\Omega)$. 
We say that~$u\in X(\Omega)$ is a weak solution of
\begin{equation}\label{probsoluzioni}
\begin{cases}
L_\mu^+ u= f &\mbox{ in } \Omega,\\
u=0 &\mbox{ in } \R^N\setminus\Omega,
\end{cases}
\end{equation}
if, for any~$v\in X(\Omega)$,
\[
\int_{[0,1]}c_{N,s}\iint_{\R^{2N}}\frac{(u(x)-u(y))(v(x)-v(y))}{|x-y|^{N+2s}}\,dx\,dy\,d\mu^+(s)= \int_\Omega f(x) v(x)\,dx.
\]

Moreover, we say that~$u\in X(\R^N)$ satisfies~$L^+_\mu u\ge f $ in~$\Omega$ in the weak sense if, for any nonnegative~$v\in X(\Omega)$
\begin{equation}\label{weak>=}
\int_{[0,1]}c_{N,s}\iint_{\R^{2N}}\frac{(u(x)-u(y))(v(x)-v(y))}{|x-y|^{N+2s}}\,dx\,dy\,d\mu^+(s)\ge \int_\Omega f(x) v(x)\,dx.
\end{equation}
\end{definition}

\begin{proof}[Proof of Theorem~\ref{weakmaxprinciple}]
Assume by contradiction that there exists a set~$E\subseteq \Omega$ with~$|E|>0$ such that~$u<0$ a.e in~$E$.

We set~$u^\pm:= \max\{\pm u,0\}$. Since~$u\ge 0$ a.e. in~$\R^N\setminus\Omega$, we have that~$u^-\equiv 0$ a.e.
in~$\R^N\setminus \Omega$. 
Moreover, we see that, for all~$x$, $y\in\R^N$,
$$ |u^-(x)-u^-(y)|\le|u(x)-u(y)|,$$
and therefore~$u^-\in X(\Omega)$.

Accordingly, we can use~$u^-$ as a test function in~\eqref{weak>=} 
(here~$f\equiv 0$) and we see that
\begin{equation}\label{78hy72}
\begin{aligned}
0&\leq \int_{[0,1]}c_{N,s}\iint_{\R^{2N}}\frac{(u(x)-u(y))(u^-(x)-u^-(y))}{|x-y|^{N+2s}}\,dx\,dy\,d\mu^+(s) \\
&= \int_{[0,1)}c_{N,s}\iint_{\R^{2N}}\frac{(u(x)-u(y))(u^-(x)-u^-(y))}{|x-y|^{N+2s}}\,dx\,dy\,d\mu^+(s)\\&\qquad\qquad
+\mu^+(\{1\})\int_\Omega \nabla u(x)\cdot\nabla u^-(x)\,dx\\
&=
\int_{[0,1)}c_{N,s}\iint_{\R^{2N}}\frac{(u^+(x)-u^+(y))(u^-(x)-u^-(y))}{|x-y|^{N+2s}}\,dx\,dy\,d\mu^+(s) \\
&\qquad\qquad-\int_{[0,1)}c_{N,s}\iint_{\R^{2N}}\frac{|u^-(x)-u^-(y)|^2}{|x-y|^{N+2s}}\,dx\,dy\,d\mu^+(s)
-\mu^+(\{1\})\int_\Omega |\nabla u^-(x)|^2\,dx
\\
&=-\int_{[0,1)}c_{N,s}\iint_{\R^{2N}}\frac{u^+(x)u^-(y)+u^+(y)u^-(x)}{|x-y|^{N+2s}}\,dx\,dy\,d\mu^+(s)\\
&\qquad\qquad-\int_{[0,1)}c_{N,s}\iint_{\R^{2N}}\frac{|u^-(x)-u^-(y)|^2}{|x-y|^{N+2s}}\,dx\,dy\,d\mu^+(s)
-\mu^+(\{1\})\int_\Omega |\nabla u^-(x)|^2\,dx\\&\le
-\int_{[0,1)}c_{N,s}\iint_{\R^{2N}}\frac{|u^-(x)-u^-(y)|^2}{|x-y|^{N+2s}}\,dx\,dy\,d\mu^+(s)
-\mu^+(\{1\})\int_\Omega |\nabla u^-(x)|^2\,dx.
\end{aligned}
\end{equation}

Now, if~$\mu^+(\{1\})>0$, we deduce from this that
$$ \int_\Omega |\nabla u^-(x)|^2\,dx\le0$$
and therefore~$u^-$ is constant, and therefore constantly equal to zero, given its trace values along~$\partial\Omega$ (recall the notation in footnote~\ref{FOO:ABNO5}). This provides a contradiction with the existence
of the above mentioned set~$E$.

Hence, we can now suppose that~$\mu^+(\{1\})=0$. In this case, we infer from~\eqref{mu0}
that~$\mu^+\big([\overline s, 1)\big)>0$, and thus by~\eqref{78hy72} that
\begin{eqnarray*}&&0\ge \int_{[\overline s, 1)}c_{N,s}\iint_{\R^{2N}}\frac{|u^-(x)-u^-(y)|^2}{|x-y|^{N+2s}}\,dx\,dy\,d\mu^+(s)\\&&\qquad
\ge \int_{[\overline s, 1)}c_{N,s}\iint_{E\times(\R^{N}\setminus\Omega)}\frac{|u^-(x)|^2}{|x-y|^{N+2s}}\,dx\,dy\,d\mu^+(s)>0.
\end{eqnarray*} This contradiction completes the proof of
Theorem~\ref{weakmaxprinciple}.
\end{proof}

We now introduce the definition of classical solutions of problem~\eqref{probsoluzioni}. To this end, we define
the norm
\begin{equation}\label{inteGRAeiuwot} \|u\|_{L^1_{\mu^+}(\R^N)}
:=\int_{(0,1)}c_{N,s} \int_{\R^N}\frac{|u(x)|}{1+|x|^{N+2s}}\,dx\,d\mu^+(s)\end{equation}
and the space
\begin{equation}\label{inteGRAeiuwot2}
C_\mu(\R^N) :=\left\{u\in C(\R^N) \mbox{ such that }\|u\|_{L^1_{\mu^+}(\R^N)}<+\infty \right\}.
\end{equation}
This definition should be compared with the one of the space~$C_s(\R^N)$
given e.g. in~\cite[formula~(2.9)]{MR4387204} when the operator~$L^+_\mu$ boils down to~$-\Delta+(-\Delta)^s$.

We stress that the integration in~\eqref{inteGRAeiuwot} is on~$s\in(0,1)$ and does not extend to~$s=0$ nor~$s=1$, since the identity map and the classical Laplacian do not require any decay assumption of the function under consideration.

We also introduce the notation
\begin{equation}\label{gammas}
 \overline{\Gamma}_N:=\max_{s\in [0,1]}\dfrac{\displaystyle\Gamma\left(\dfrac{N+2s}{2}\right)}{\displaystyle\Gamma(2-s)}
%\qquad{\mbox{and}}\qquad
%\underline{\Gamma}_N:=
%\min_{s\in [0,1]}\dfrac{\displaystyle\Gamma\left(\dfrac{N+2s}{2}\right)}{\displaystyle\Gamma(2-s)}.
\end{equation} 

Our next result shows that the space~$C_\mu(\R^N)$ is the proper one to handle classical solutions for the operator~$L^+_\mu$. Indeed, for~$u\in C_\mu(\R^N)\cap C^2(\Omega)$, one can compute~$L_\mu^+ u$ pointwise, according to the following statement:

\begin{proposition}\label{propcompute}
Let~$u\in C_\mu(\R^N)\cap C^2(\Omega)$.
Then, $L^+_\mu u \in L^\infty_{loc}(\Omega)$ and, for all~$x\in\Omega$
such that~$B_r(x)\subset\Omega$ for some~$r\in(0,1)$,
$$ |L^+_\mu u(x)| \le
C\left( \|D^2 u\|_{L^\infty(B_r(x))} +|u(x)|
+\|u\|_{L^1_{\mu^+}(\R^N)}\right),
$$
for some~$C>0$, depending on~$N$, $\Omega$, $r$ and~$\mu^+$.
\end{proposition}

\begin{proof}
Fix~$x\in \Omega$ and let~$r\in(0,1)$ be such that~$B_r(x)\subset \Omega$. Recalling the definitions in~\eqref{superpositionoperator+} and~\eqref{AMMU0}, we can write
\begin{equation}\label{pointwise1}
\begin{aligned}
L_\mu^+ u(x)&=\int_{(0,1)}c_{N,s}\int_{\R^N} \frac{2u(x) - u(x+y)-u(x-y)}{|y|^{N+2s}}\, dy\,d\mu^+(s)\\&\qquad\qquad
+\mu^+(\{0\})u(x)-\mu^+(\{1\})\Delta u(x)\\
&=\int_{(0,1)}c_{N,s}\int_{B_r} \frac{2u(x) - u(x+y)-u(x-y)}{|y|^{N+2s}}\, dy\,d\mu^+(s) \\
&\qquad\qquad +\int_{(0,1)}c_{N,s}\int_{\R^N\setminus B_r} \frac{2u(x) - u(x+y)-u(x-y)}{|y|^{N+2s}}\, dy\,d\mu^+(s)
\\&\qquad\qquad+\mu^+(\{0\})u(x)-\mu^+(\{1\})\Delta u(x).
\end{aligned}
\end{equation}
Now, when~$y\in B_r$, a second order Taylor expansion for both~$u(x+y)$ and~$u(x-y)$ gives that
\begin{equation}\label{tyuriefhdjbdnsy4u56473467}
\begin{split}
\Bigg|\int_{(0,1)}c_{N,s}&\int_{B_r} \frac{2u(x) - u(x+y)-u(x-y)}{|y|^{N+2s}}\, dy\,d\mu^+(s)\Bigg| \\
&\leq 2\|D^2 u\|_{L^\infty(B_r(x))}\int_{(0,1)}c_{N,s}\int_{B_r} \frac{dy}{|y|^{N+2s-2}} \,d\mu^+(s) \\
&=\omega_{N-1}\|D^2 u\|_{L^\infty(B_r(x))}\int_{(0,1)}\frac{c_{N,s} r^{2-2s}}{1-s} \,d\mu^+(s),
 \end{split}
\end{equation}
where~$\omega_{N-1}$ is the surface area of the unit sphere in~$\R^N$.

We now recall~\eqref{defcNs} and see that
\begin{eqnarray*}
\frac{c_{N,s}}{1-s} = \frac{2^{2s-1}s\,\Gamma\left(\frac{N+2s}{2}\right)}{\pi^{N/2}\,\Gamma(2-s)}\le
\frac{2\overline{\Gamma}_N}{\pi^{N/2}},
\end{eqnarray*} 
where the notation in~\eqref{gammas} has been used.
Exploiting this information into~\eqref{tyuriefhdjbdnsy4u56473467}, we thus obtain that
\begin{equation}\label{tyuriefhdjbdnsy4u564734672}\begin{split}&
\Bigg|\int_{(0,1)}c_{N,s}\int_{B_r} \frac{2u(x) - u(x+y)-u(x-y)}{|y|^{N+2s}}\, dy\,d\mu^+(s)\Bigg|\\&\qquad\qquad\le 
\frac{2\overline{\Gamma}_N\omega_{N-1}\mu^+([0,1])}{\pi^{N/2}}
\|D^2 u\|_{L^\infty(B_r(x))}.\end{split}
\end{equation}

Furthermore, up to a suitable change of variables, we can write
\begin{equation}\label{pointwise3}
\begin{split}&
\Bigg|\int_{(0,1)}c_{N,s}\int_{\R^N\setminus B_r} \frac{2u(x) - u(x+y)-u(x-y)}{|y|^{N+2s}}\, dy\,d\mu^+(s)\Bigg| \\
&\quad\leq 2|u(x)|\int_{(0,1)}c_{N,s}\int_{\R^N\setminus B_r} \frac{dy}{|y|^{N+2s}}\,d\mu^+(s)
+2\int_{(0,1)}c_{N,s}\int_{\R^N\setminus B_r} \frac{|u(x+y)|}{|y|^{N+2s}}\,dy\,d\mu^+(s).
\end{split}
\end{equation} 
We have that
\begin{equation*}
\int_{(0,1)}
c_{N,s}\int_{\R^N\setminus B_r} \frac{dy}{|y|^{N+2s}}\,d\mu^+(s)
=\omega_{N-1}\int_{(0,1)} \frac{c_{N,s}}{2s r^{2s}}\,d\mu^+(s) .
\end{equation*}
{F}rom~\eqref{defcNs} we see that
$$
\frac{c_{N,s}}{s} = \frac{2^{2s-1}(1-s)\,\Gamma\left(\frac{N+2s}{2}\right)}{\pi^{N/2}\,\Gamma(2-s)}\le
\frac{2\overline{\Gamma}_N}{\pi^{N/2}},$$
and thus
\begin{equation}\label{pointwise4}
\int_{(0,1)}
c_{N,s}\int_{\R^N\setminus B_r} \frac{dy}{|y|^{N+2s}}\,d\mu^+(s)
\le \frac{\overline{\Gamma}_N \omega_{N-1}\mu^+((0,1))
}{r^2\pi^{N/2}}
.\end{equation}

Now we claim that there exists~$C>0$
such that, for all~$y\in \R^N\setminus B_r$ and all~$s\in [0,1]$,
\begin{equation}\label{cftgvbhbhjjn}
\frac{1+|x+y|^{N+2s}}{|y|^{N+2s}} \le C.
\end{equation}
To check this, we define the function
\[
f(y,s):=\frac{1+|x+y|^{N+2s}}{|y|^{N+2s}}
\]
and we observe that
\[
\lim_{|y|\to +\infty}f(y,s)=1 \quad \mbox{uniformly for any }s\in [0,1].
\]
As a consequence,
there exists~$M>1$ such that, if~$|y|\geq M$, then~$f(y,s)\leq 2$.

Also, if~$y\in B_M\setminus B_r$, then
\[
f(y,s)\leq \frac{1+(|x|+M)^{N+2s}}{r^{N+2s}}\leq \frac{1+(|x|+M)^{N+2}}{r^{N+2}}.
\]
These considerations establish the claim in~\eqref{cftgvbhbhjjn}.

{F}rom the bound in~\eqref{cftgvbhbhjjn} we thus obtain that
\begin{eqnarray*}
\int_{(0,1)} c_{N,s}
\int_{\R^N\setminus B_r} \frac{|u(x+y)|}{|y|^{N+2s}}\,dy\,d\mu^+(s)
\le C\int_{(0,1)} c_{N,s}
\int_{\R^N\setminus B_r}\frac{|u(x+y)|}{1+|x+y|^{N+2s}}\,dy\,d\mu^+(s)
.\end{eqnarray*}
Hence, changing variable~$z:=x+y$ and recalling that~$u\in C_\mu(\R^N)$,
we find that
\begin{equation*}
\int_{(0,1)}
c_{N,s}\int_{\R^N\setminus B_r} \frac{|u(x+y)|}{|y|^{N+2s}}\,dy\,d\mu^+(s)
\leq C\int_{(0,1)}
c_{N,s}\int_{\R^N} \frac{|u(z)|}{1+|z|^{N+2s}}\,dz\,d\mu^+(s)=C\|u\|_{L^1_{\mu^+}(\R^N)}.
\end{equation*}

Plugging this and~\eqref{pointwise4} into~\eqref{pointwise3}, we thus
obtain that
\begin{equation*}\begin{split}&
\Bigg|\int_{(0,1)}c_{N,s}\int_{\R^N\setminus B_r} \frac{2u(x) - u(x+y)-u(x-y)}{|y|^{N+2s}}\, dy\,d\mu^+(s)\Bigg| \\
&\qquad\leq\frac{2\overline{\Gamma}_N\omega_{N-1}\mu^+((0,1))|u(x)|
}{r^2\pi^{N/2}}
+2C\|u\|_{L^1_{\mu^+}(\R^N)}.
\end{split}\end{equation*}

Combining this with~\eqref{pointwise1} and~\eqref{tyuriefhdjbdnsy4u564734672}, we conclude that
\begin{eqnarray*}
|L_\mu^+ u(x)|&\le& \frac{2\overline{\Gamma}_N\omega_{N-1}\mu^+((0,1))}{\pi^{N/2}}
\left( \|D^2 u\|_{L^\infty(B_r(x))} +\frac{|u(x)|
}{r^2}\right)
+2C\|u\|_{L^1_{\mu^+}(\R^N)}\\&&\qquad+\mu^+(\{0\})|u(x)|+\mu^+(\{1\})|\Delta u(x)|,
\end{eqnarray*}
which gives the desired result.
\end{proof}

We can now give the following:

\begin{definition}\label{xderfcvgtygbhu}
Let~$f:\Omega\to \R$ be any measurable function. We say that a function~$u\in C_\mu(\R^N)\cap C^2(\Omega)$ is a classical solution 
of problem~\eqref{probsoluzioni} if~$L_\mu^+ u(x)=f(x)$ for any~$x\in\Omega$ and~$u(x)=0$ for any~$x\in \R^N\setminus \Omega$.
\end{definition}

\begin{proof}[Proof of Theorem~\ref{strongmaximumprinciple}]
If~$\mu^+((0,1))=0$, then the problem boils down to the classical
maximum principle for the Laplace operator, see e.g.~\cite[Theorem~2 on page~329 and Theorem~4 on page~333]{MR1625845}, hence we can suppose
that~$\mu^+((0,1))>0$.

We want to prove that~$u\geq 0$ in~$\Omega$. For this, we assume by contradiction that there exists some~$\overline{x}\in \Omega$ 
such that~$u(\overline{x})<0$.
Then, we prove that
\begin{equation}\label{CANNO}
{\mbox{$u$ cannot have an interior minimum, unless it is constant.}}
\end{equation}
To this end, we assume by contradiction that~$u$ is not constant
and that it possesses a minimum at a point~$x_0\in \Omega$. This implies that~$\Delta u(x_0)\ge 0$ and $u(x_0)\leq u(\overline{x})<0$.
Moreover, for all~$s\in(0,1)$,
\begin{eqnarray*}
(-\Delta)^s u(x_0)= c_{N,s}\int_{\R^N} \frac{2u(x_0) - u(x_0+y)-u(x_0-y)}{|y|^{N+2s}}\, dy\le 0.
\end{eqnarray*}

Therefore, we have that
\begin{eqnarray*}&&
0\le L^+_\mu u(x_0) = 
\int_{[0, 1]} (-\Delta)^s u(x_0)\, d\mu^+(s)\\&&\qquad
=-\mu^+(\{1\})\Delta u(x_0)+\int_{(0, 1)} (-\Delta)^s u(x_0)\, d\mu^+(s)
+\mu^+(\{0\}) u(x_0)\\&&\qquad\le \int_{(0, 1)} (-\Delta)^s u(x_0)\, d\mu^+(s)\\&&\qquad
=\int_{(0, 1)}\left[ c_{N,s}\int_{\R^N} \frac{2u(x_0) - u(x_0+y)-u(x_0-y)}{|y|^{N+2s}}\, dy\right]\, d\mu^+(s)
\le0.
\end{eqnarray*}

Consequently, for almost every~$y\in\R^N$,
$$\int_{(0, 1)}c_{N,s} \frac{2u(x_0) - u(x_0+y)-u(x_0-y)}{|y|^{N+2s}}\,d\mu^+(s)=0.$$
As a result, since~$u$ is not constant (and therefore~$2u(x_0) - u(x_0+y)-u(x_0-y)$
is not identically null), we find that, for some~$y\in\R^N\setminus\{0\}$,
$$\int_{(0, 1)} \frac{c_{N,s}}{|y|^{2s}}\,d\mu^+(s)=0.$$
Since the integrand is strictly positive, we have thus derived a contradiction
and we have completed the proof of~\eqref{CANNO}.

Now, the statement in~\eqref{CANNO} and the fact that~$u(\overline{x})<0$ are in contradiction with
the assumption that~$u\ge0$ in~$\R^N\setminus\Omega$.
Hence, we conclude that~$u\ge0$ in~$\Omega$, which is~\eqref{2uge0Omega}.

Now we assume that there exists~$x_0\in \Omega$ such that~$u(x_0)=0$.
Since we already know that~$u(x)\geq 0$ for any~$x\in \R^N$, it follows that~$x_0$ is a minimum point
for~$u$. Hence, by~\eqref{CANNO},
we have that~$u$ vanishes identically. This completes the proof of Theorem~\ref{strongmaximumprinciple}.
\end{proof}

\section{The eigenvalue problem}\label{sectioneigenvalues}
We start this section by introducing some preliminary notations. 
We recall that the space~$X(\Omega)$ has been defined in~\eqref{XOmegadefn}. 

The weak formulation of the eigenvalue problem~\eqref{eigenvalueproblem} is given by
\begin{equation}\label{weakeigenvalue}
\begin{split}
&\int_{[0,1]}c_{N,s}\iint_{\R^{2N}}\frac{(u(x)-u(y))(v(x)-v(y))}{|x-y|^{N+2s}}\,dx\,dy\,d\mu^+(s)\\
&\quad -  \int_{[0,\overline s)}c_{N,s}\iint_{\R^{2N}}\frac{(u(x)-u(y))(v(x)-v(y))}{|x-y|^{N+2s}}\,dx\,dy\,d\mu^-(s) \\
&= \lambda\int_\Omega u(x) v(x)\, dx \qquad\mbox{ for any } v\in X(\Omega),
\end{split}
\end{equation}
where the ``abuse of notation'' in~\eqref{ocmrgexpresfyrj5784i} is assumed.

We recall that if there exists a nontrivial solution~$u\in X(\Omega)$ of~\eqref{weakeigenvalue}, then~$\lambda\in\R$ is called an eigenvalue of the operator~$L_\mu$. Any solution~$u\in X(\Omega)$ is called an eigenfunction associated with the eigenvalue~$\lambda$.

Furthermore, let~$I:X(\Omega)\to\R$ be the functional defined as
\begin{equation}\label{Ifunctional}
I(u):= \frac12 \int_{[0, 1]} [u]^2_s\, d\mu^+(s) - \frac12\int_{[0, \overline s)} [u]^2_s\, d\mu^-(s) = \frac12 \|u\|^2,
\end{equation}
where~$\|u\|$ is the norm given in~\eqref{norma+-}.

We recall the following result proved in~\cite[Lemma~3.1]{DPLSVlog}:

\begin{lemma}\label{lemmino1}
Let~$\mu$ satisfy~\eqref{mu0}, \eqref{mu1} and~\eqref{mu2}. 
%Let~$s_\sharp$ be as in~\eqref{scritico} and~$N> 2s_{\sharp}$.

Let~$X_0$ be a nonempty, weakly closed subspace of~$X(\Omega)$ and
\[
\mathscr M:=\big\{ u\in X_0\;{\mbox{ s.t. }}\; \|u\|_{L^2(\Omega)}=1\big\}.
\]

Then, there exists~$\gamma_0>0$, depending only on~$N$ and~$\Omega$,
such that if~$\gamma\in[0,\gamma_0]$ the following statements hold true.

There exists~$u_0\in\mathscr M$ such that
\begin{equation}\label{minimo1}
\min_{u\in\mathscr M} I(u) = I(u_0)>0.
\end{equation}

In addition, for any~$v\in X_0$,
\begin{equation}\label{minimo2}
\begin{split}
&\int_{[0,1]}c_{N,s}\iint_{\R^{2N}}\frac{(u_0(x)-u_0(y))(v(x)-v(y))}{|x-y|^{N+2s}}\,dx\,dy\,d\mu^+(s)\\
&\quad -  \int_{[0,\overline s)}c_{N,s}\iint_{\R^{2N}}\frac{(u_0(x)-u_0(y))(v(x)-v(y))}{|x-y|^{N+2s}}\,dx\,dy\,d\mu^-(s) = 2I(u_0)\int_\Omega u_0(x) v(x)\, dx.
\end{split}
\end{equation}
\end{lemma}

The next two results provide some useful properties of the eigenfunctions and will be useful when dealing with problem~\eqref{eigenvalueproblem}.
%%We point out that this corresponds to the case in which the measure~$\mu$ consists only of positive contributions.

\begin{lemma}\label{lemmino2}
Let~$\lambda_1$ and~$\lambda_2$ be two different eigenvalues of~\eqref{eigenvalueproblem} and let~$e_1$, $e_2\in X(\Omega)$ be associated eigenfunctions.

Then,
\[
\int_{[0,1]}c_{N,s}\iint_{\R^{2N}}\frac{(e_1(x)-e_1(y))(e_2(x)-e_2(y))}{|x-y|^{N+2s}}\,dx\,dy\,d\mu(s) =\int_\Omega e_1(x) e_2(x)\, dx=0.
\]
\end{lemma}

\begin{proof}
Up to a normalization,
we can assume that~$\|e_1\|_{L^2(\Omega)}=\|e_2\|_{L^2(\Omega)}=1$.

Testing the equation in~\eqref{weakeigenvalue} for~$e_1$ against~$e_2$, we get
\begin{equation}\label{e1vse2}
\int_{[0,1]}c_{N,s}\iint_{\R^{2N}}\frac{(e_1(x)-e_1(y))(e_2(x)-e_2(y))}{|x-y|^{N+2s}}\,dx\,dy\,d\mu(s)
=\lambda_1 \int_\Omega e_1(x) e_2(x)\, dx.
\end{equation}
Similarly, testing~\eqref{weakeigenvalue} for~$e_2$ against~$e_1$, we obtain that
\begin{equation}\label{e2vse1}
\int_{[0,1]}c_{N,s}\iint_{\R^{2N}}\frac{(e_2(x)-e_2(y))(e_1(x)-e_1(y))}{|x-y|^{N+2s}}\,dx\,dy\,d\mu(s) =\lambda_2 \int_\Omega e_1(x) e_2(x)\, dx.
\end{equation}
Combining~\eqref{e1vse2} and~\eqref{e2vse1}, we infer that
\[
(\lambda_1-\lambda_2) \int_\Omega e_1(x) e_2(x)\, dx=0.
\]
As a consequence, since~$\lambda_1\neq\lambda_2$,
\[
\int_\Omega e_1(x) e_2(x)\, dx=0.
\]
Using this information into~\eqref{e1vse2} (or equivalently~\eqref{e2vse1}), 
we complete the proof.
\end{proof}

\begin{lemma}\label{lemmino3}
Let~$\lambda$ be an eigenvalue of~\eqref{eigenvalueproblem} and let~$e\in X(\Omega)$ be an associated eigenfunction.

Then,
\[
\int_{[0,1]}c_{N,s}\iint_{\R^{2N}}\frac{|e(x)-e(y)|^2}{|x-y|^{N+2s}}\,dx\,dy\,d\mu(s)=\lambda \|e\|^2_{L^2(\Omega)}.
\]
\end{lemma}

\begin{proof}
The desired result plainly follows by testing the equation~\eqref{weakeigenvalue} for~$e$ against itself.
\end{proof}
We now establish some properties of eigenvalues and eigenfunctions for
the operator~$L_{\mu}$ of classical flavour.
The proofs are inspired by~\cite[Proposition~9]{MR3002745},
where the case of the fractional Laplacian~$(-\Delta)^s $ is taken into account.
%%We provide here below the precise statements and the details of the proofs for the facility of the reader.

We first prove that the set of eigenvalues of problem~\eqref{eigenvalueproblem} consists of a sequence~$\lambda_k$ which diverges as~$k$ goes to infinity:

\begin{proposition}\label{propde}
Let~$\mu$ satisfy~\eqref{mu0}, \eqref{mu1} and \eqref{mu2}. Take $N> 2\overline s$.
%Let~$s_\sharp$ be as in~\eqref{scritico} and~$N> 2s_{\sharp}$.

Then, there exists~$\gamma_0>0$, depending only on~$N$ and~$\Omega$, such that if~$\gamma\in[0,\gamma_0]$ the following statements hold true.

There exists a sequence~$\lambda_k$ of eigenvalues of problem~\eqref{eigenvalueproblem} such that
\begin{equation}\label{unboundedsequence}
0<\lambda_{1} \le \lambda_2\le \cdots\le \lambda_k \le\cdots
\end{equation}
and
\begin{equation}\label{unboundedeigenvalues}
\lim_{k\to +\infty} \lambda_k = +\infty.
\end{equation}

In addition, for any~$k\in \N$ with~$k\ge1$, the eigenvalues~$\lambda_{k+1}$
are given recursively by
\begin{equation}\label{kautovalore}
\lambda_{k+1}=\min_{u\in E_{k+1}\setminus \{0\}} \dfrac{\displaystyle\int_{[0, 1]} [u]^2_s \, d\mu^+(s)-\int_{[0, \overline s)} [u]^2_s \, d\mu^-(s)}{\displaystyle\int_\Omega |u(x)|^2 \,dx},
\end{equation}
where~$E_1:= X(\Omega)$ and, for all~$k\ge1$,
\begin{equation}\label{Ek+1definition}
E_{k+1}:=\big\{u\in X(\Omega) \;{\mbox{ s.t. }}\; \langle u, e_j\rangle=0 \ \mbox{ for all } j=1,\dots,k\big\}.
\end{equation}
%Here above, $e_1:=e_{\mu^+}$, with $e_{\mu^+}$ given by  Proposition~\ref{lambda1simple}.

Moreover, for any~$k\in\N$ with~$k\ge1$,
the function~$e_{k+1}\in E_{k+1}$ is an eigenfunction corresponding to the eigenvalue~$\lambda_{k+1}$ and achieves the minimum in~\eqref{kautovalore}.
\end{proposition}

\begin{proof}
The existence of the first eigenvalue is guaranteed by~\cite[Theorem~1.1]{DPLSVlog}. Hence, we now focus on the statements
for~$\lambda_{k+1}$ with~$k\ge1$.

We observe that 
\begin{equation}\label{weaklyclosed}
E_{k+1} \mbox{ is a weakly closed subspace of~$X(\Omega)$}.
\end{equation}
Indeed, let~$u_n\subset E_{k+1}$ be such that 
\[
\lim_{n\to +\infty} \langle u_n, v\rangle = \langle u, v\rangle \quad\mbox{ for any } v\in X(\Omega).
\]
Now, for any~$j=1,\dots,k$, we have that
\[
0= \lim_{n\to +\infty} \langle u_n, e_j\rangle = \langle u, e_j\rangle,
\]
which gives that~$u\in E_{k+1}$. This proves~\eqref{weaklyclosed}.

In light of~\eqref{weaklyclosed}, we can apply Lemma~\ref{lemmino1} with
$$X_0=: E_{k+1}\qquad{\mbox{and}}\qquad
\mathscr M:=\big\{u\in E_{k+1} \;{\mbox{ s.t. }}\; \|u\|_{L^2(\Omega)}=1\big\}.$$
Accordingly, formula~\eqref{minimo1} gives that the minimum in~\eqref{kautovalore} is attained at some~$e_{k+1}\in E_{k+1}$, with~$\|e_{k+1}\|_{L^2(\Omega)}=1$.
%Also, by taking~$\mu^-\equiv 0$ in~\eqref{minimo2}, we have that
%\begin{equation}\label{eigenfunctioninEk+1}
%\begin{split}
%&\int_{[0,1]}c_{N,s}\iint_{\R^{2N}}\frac{(e_{k+1}(x)-e_{k+1}(y))(v(x)-v(y))}{|x-y|^{N+2s}}\,dx\,dy\,d\mu^+(s)\\
%&\quad= \lambda_{k+1}\int_\Omega e_{k+1}(x) v(x)\, dx \qquad\mbox{ for any } v\in E_{k}.
%\end{split}
%\end{equation}

Now, by~\eqref{Ek+1definition}, we deduce that~$E_{k+1}\subseteq E_k\subseteq X(\Omega)$, and therefore~$\lambda_k\le\lambda_{k+1}$.  Thus, \eqref{unboundedsequence},  is proved.

%For this, let~$e_{1}$, $e_2 \in X(\Omega)$ be the eigenfunctions
%associated with~$\lambda_1$ and~$\lambda_2$ respectively,
%and assume by contradiction that~$\lambda_{1}=\lambda_2$.  Thus, $e_2$ is an eigenfunction for~$\lambda_{1}$ as well.  
%Now, by Proposition~\ref{lambda1simple} we know that~$\lambda_{1}$ is simple,
%and therefore there exists~$c\in\R\setminus\{0\}$ such that~$e_2 = c\, e_{1}$. Moreover, since~$e_2\in E_2$, we have that
%\begin{equation*}
%0= \langle e_2, e_{1}\rangle_+ = c \langle e_{1}, e_{1}\rangle_+ = c\int_{[0, 1]} [e_{1}]^2_s \, d\mu^+(s).
%\end{equation*}
%Thus, since~$c\neq0$, we infer that~$e_{1}=0$, which is a contradiction. This
%completes the proof of~\eqref{unboundedsequence}.

In order to show that~$\lambda_{k+1}$ defined in~\eqref{kautovalore} is an eigenvalue of problem~\eqref{eigenvalueproblem} with corresponding eigenfunction~$e_{k+1}$, we need to show that
\begin{equation}\label{eigenfunctioninX}
\begin{split}
&\int_{[0,1]}c_{N,s}\iint_{\R^{2N}}\frac{(e_{k+1}(x)-e_{k+1}(y))(v(x)-v(y))}{|x-y|^{N+2s}}\,dx\,dy\,d\mu(s)\\
&\quad= \lambda_{k+1}\int_\Omega e_{k+1}(x) v(x)\, dx \qquad\mbox{ for any } v\in X(\Omega).
\end{split}
\end{equation}
To this end, we perform an inductive argument. 
We observe that if~$k=0$, then the desired result follows from~\cite[Theorem~1.1]{DPLSVlog}. Now, we assume that~\eqref{eigenfunctioninX} holds true
for~$k$ and we prove that it is true for~$k+1$ as well.

We write
\[
X(\Omega)= \mbox{span}\{e_1,\dots,e_k\}\oplus E_{k+1}.
\]
Accordingly, for any~$v\in X(\Omega)$, there exist~$a_1,\dots,a_k\in\R$ and~$v_2\in E_{k+1}$ such that
\[
v=v_1 + v_2 = \sum_{j=1}^k a_j e_j + v_2.
\]
Thus, by testing~\eqref{weakeigenvalue} for $u:=e_{k+1}$ against~$v_2=v-v_1\in E_{k+1}$, we get that
\begin{equation}\label{tfygubhinjokm}
\begin{split}
&\langle e_{k+1}, v\rangle -\lambda_{k+1}\int_\Omega e_{k+1}(x) v(x)\, dx= \langle e_{k+1}, v_1\rangle -\lambda_{k+1}\int_\Omega e_{k+1}(x) v_1(x)\, dx\\
&\qquad=\sum_{j=1}^k a_j  \langle e_{k+1}, e_j\rangle -\lambda_{k+1}\sum_{j=1}^k a_j \int_\Omega e_{k+1}(x) e_j(x)\, dx.
\end{split}
\end{equation}

Moreover, in light of the inductive assumption, we can test~\eqref{eigenfunctioninX} for~$e_j$ against~$e_{k+1}$ for any~$j=1,\dots, k$.  Thus, recalling~\eqref{Ek+1definition}, we obtain that, for any~$j=1, \dots, k$,
\[
\lambda_j \int_\Omega e_j(x)e_{k+1}(x) \, dx = \langle e_j, e_{k+1}\rangle= 0 .
\]
Since~$\lambda_j>0$ we thereby find that, for any~$j=1, \dots, k$,
\begin{equation*}
\int_\Omega  e_j(x)e_{k+1}(x)\, dx = \langle e_j, e_{k+1}\rangle= 0 .
\end{equation*}
{F}rom this and~\eqref{tfygubhinjokm}, we deduce that
\[
\langle e_{k+1}, v\rangle =\lambda_{k+1}\int_\Omega e_{k+1}(x) v(x)\, dx \quad\mbox{ for any }v\in X(\Omega),
\]
and therefore~\eqref{eigenfunctioninX} holds true for~$k+1$. This concludes
the inductive argument.

We now focus on the proof of~\eqref{unboundedeigenvalues}. Let~$j$, $k\in\N$ with~$j\ne k$.  Without loss of generality, we assume that~$k>j$. By testing~\eqref{weakeigenvalue} for the eigenfunction~$e_k$ against~$e_j$, we see that
\begin{equation}\label{uygbhijn}
\langle e_k, e_j\rangle = \lambda_k \int_\Omega e_k(x) e_j(x)\, dx.
\end{equation}
Moreover, since~$e_k\in E_k$ and~$k-1\ge j$, by~\eqref{Ek+1definition} we deduce that~$\langle e_k, e_j\rangle_+ =0$. Since~$\lambda_k>0$,
from this fact and~\eqref{uygbhijn} we thus infer that
\begin{equation}\label{ftygubhnj}
\langle e_k, e_j\rangle=0 = \int_\Omega e_k(x) e_j(x)\, dx \quad\mbox{ for any } k, j\in\N,  \, k\ne j.
\end{equation}

Now, we assume by contradiction that there exists~$c\in\R$ such that 
\begin{equation}\label{contradictionbounded}
\lim_{k\to +\infty} \lambda_k = c.
\end{equation}
In particular, $\lambda_k$ is a bounded sequence in~$\R$. Moreover, by Lemma~\ref{lemmino3} and~\eqref{norma+-}, we get that
\[
\|e_k\|^2 = \int_{[0, 1]} [e_k]^2_s \, d\mu(s) = \lambda_k \|e_k\|^2_{L^2(\Omega)} = \lambda_k.
\]
{F}rom this and~\eqref{contradictionbounded}, we infer that~$e_k$ is a bounded sequence in~$X(\Omega)$. Hence, by Propositions~\ref{embeddingsXOmega} and~\ref{propnormeequiv}, there exist a subsequence~$e_{k_j}$ and~$e_0\in L^2(\Omega)$ such that
\[
{\mbox{$e_{k_j}\to e_0$ in~$L^2(\Omega)$ as~$j\to+\infty$.}} 
\]
In particular, $e_{k_j}$ is a Cauchy sequence in~$L^2(\Omega)$, that is for any~$\varepsilon>0$ there exists~$N_\varepsilon >0$ such that,
if~$i$, $ j\ge N_\varepsilon$, then~$\|e_{k_j} - e_{k_i}\|_{L^2(\Omega)}\le\varepsilon$.

On the other hand, by~\eqref{ftygubhnj}, we have that,
for any~$i$, $j\in\N$ with~$i\ne j$, 
\[
\|e_{k_j} - e_{k_i}\|^2_{L^2(\Omega)} = \int_\Omega (e_{k_j} - e_{k_i})(e_{k_j} - e_{k_i}) \,dx = \int_{\Omega} |e_{k_j}|^2 \,dx + \int_{\Omega} |e_{k_i}|^2 \,dx = 2.
\]
This is a contradiction, which establishes~\eqref{unboundedeigenvalues}.

It remains to show that 
\begin{equation}\label{claimogniautovalore}
\mbox{ any eigenvalue of problem~\eqref{eigenvalueproblem} can be written as in~\eqref{kautovalore}.}
\end{equation}
To this aim, we assume by contradiction that there exists an eigenvalue~$\lambda\ne\lambda_k$ for any~$k\in\N$ with corresponding eigenvalue~$e$, normalized in such a way that~$\|e\|_{L^2(\Omega)}=1$. Thus, by Lemma~\ref{lemmino3} and~\eqref{Ifunctional}, 
\begin{equation}\label{gyvbhuijnmk}
2 I(e) = \|e\|^2 = \lambda.
\end{equation}
This, together with Lemma~\ref{lemmino1}, guarantees that
\[
\lambda_1 = 2 I(e_{1})\le 2 I(e) =\lambda.
\]

Now, we show that there exists~$k\in\N$ such that 
\begin{equation}\label{lambdaminorestretto}
\lambda_k < \lambda <\lambda_{k+1}.
\end{equation}
Indeed, by~\eqref{unboundedeigenvalues} there exists~$\widetilde k\in\N$ such that~$\lambda<\lambda_{\widetilde k}$. Moreover, by~\eqref{unboundedsequence} there exists~$h< \widetilde k$ such that
\[
\lambda_{\widetilde k -h}<\lambda< \lambda_{\widetilde k -h+1}.
\]
By taking~$k:=\widetilde k -h$, we get~\eqref{lambdaminorestretto}.

Now, we observe that
\begin{equation}\label{claimnotinEk+1}
e\notin E_{k+1}.
\end{equation}
Indeed, if by contradiction~$e\in E_{k+1}$, then by~\eqref{kautovalore} and~\eqref{gyvbhuijnmk}, we would have that~$ \lambda=2I(e)\ge\lambda_{k+1}$, which contradicts~\eqref{lambdaminorestretto}.

In light of~\eqref{claimnotinEk+1}, we have that
there exists~$j\in\{1,\dots, k\}$ such that~$\langle e, e_j\rangle \ne 0$, which is in contradiction with Lemma~\ref{lemmino2}. This 
completes the proof of~\eqref{claimogniautovalore}, as desired.
\end{proof}

In the next result we show that the
sequence of eigenfunctions~$e_k$ corresponding to the eigenvalues~$\lambda_k$ forms an orthonormal basis of~$L^2(\Omega)$ and an orthogonal basis of~$X(\Omega)$.

\begin{proposition}\label{propf}
Let~$\mu$ satisfy~\eqref{mu0}, \eqref{mu1} and~\eqref{mu2}. Take $N> 2\overline s$.
%Let~$s_\sharp$ be as in~\eqref{scritico} and~$N> 2s_{\sharp}$.
Let~$\lambda_k$ be the sequence of eigenvalues given by Proposition~\ref{propde}. 

Then, the sequence~$e_k$ of the corresponding eigenfunctions is an orthonormal basis of~$L^2(\Omega)$ and an orthogonal basis of~$X(\Omega)$.
\end{proposition}

\begin{proof}
The orthogonality of the eigenfunctions in~$L^2(\Omega)$ and~$X(\Omega)$
is a consequence of the claim in~\eqref{ftygubhnj}. Moreover, 
up to a renormalization, 
we can suppose that~$\|e_k\|_{L^2(\Omega)}=1$ for any~$k\in \N$. Thus, from now on we focus on the proof of the fact that the sequence of eigenvalues~$e_k$ is a basis for both~$X(\Omega)$ and~$L^2(\Omega)$.

We start by showing that the sequence~$e_k$ is a basis of~$X(\Omega)$. To this end, we claim that
\begin{equation}\label{claimbaseX}
\mbox{if } v\in X(\Omega)\mbox{ is such that }\langle v, e_k\rangle=0 \mbox{ for any } k\in\N, \mbox{ then } v\equiv 0.
\end{equation}
Suppose by contradiction that there exists~$v\in X(\Omega)$, with~$v\not\equiv 0$, such that
\begin{equation}\label{xdrfvgtygbnhjuj}
\langle v, e_k\rangle=0 \mbox{ for any } k\in\N.
\end{equation} 
Since~$v\not \equiv 0$, up to a normalization we can assume that~$\|v\|_{L^2(\Omega)}=1$. Thus, by~\eqref{unboundedeigenvalues} and~\eqref{kautovalore} there exists~$k\in \N$ sufficiently large such that
\[
\frac{\displaystyle \int_{[0,1]}[v]_s^2\,d\mu^+(s) -\int_{[0, \overline s)} [v]^2_s \, d\mu^-(s)}{ \displaystyle\|v\|_{L^2(\Omega)}}
<\lambda_{k+1}.
\]
Hence, $v$ does not belong to~$E_{k+1}$. Therefore, by~\eqref{Ek+1definition}, there exists~$j\in \{1,\dots,k\}$ such that~$\langle v, e_j\rangle\ne 0$, which is in contradiction with~\eqref{xdrfvgtygbnhjuj}, and the proof of~\eqref{claimbaseX} is complete.

To establish the fact that
\begin{equation}\label{baseX}
e_k \mbox{ is a basis of } X(\Omega)
\end{equation}
we need to employ standard Fourier analysis. We discuss the details of this construction in Appendix~\ref{appendixFourier}.

We now show that~$e_k$ is a basis for~$L^2(\Omega)$. To this end, let~$v\in L^2(\Omega)$ and, for any~$\varepsilon>0$, let~$v_\varepsilon\in 
C^\infty_0(\Omega)$ be such that
\begin{equation}\label{xdfcvgyuhnjko}
\|v_\varepsilon-v\|_{L^2(\Omega)}\leq \varepsilon.
\end{equation}

We observe that~$v_\varepsilon\in X(\Omega)$. Indeed, using Proposition~\ref{propnormeequiv} and Lemma~\ref{nons} with~$s_1:=s$ and~$s_2:=1$, we get that
\[
\|v_\varepsilon\|^2\le \|v_\varepsilon\|_X^2=\int_{[0,1]}[v_\varepsilon]_s^2\,d\mu^+(s)\le C\int_{[0,1]}[v_\varepsilon]_1^2\,d\mu^+(s)
=C\|\nabla v_\varepsilon\|_{L^2(\Omega)}^2\mu^+\big([0,1]\big),
\] for some~$C>0$ depending only on~$N$ and~$\Omega$.

Thus, by~\eqref{baseX}, there exist~$k_\varepsilon\in\N$ and~$w_\varepsilon \in \mbox{span}\{e_1,\dots,e_{k_\varepsilon}\}$ such that
\[
\|v_\varepsilon-w_\varepsilon\|_X\leq \varepsilon.
\]
This and Proposition~\ref{embeddingsXOmega} give that
\begin{equation}\label{xdfcvgyuhnjkoxdftyujk}
\|v_\varepsilon-w_\varepsilon\|_{L^2(\Omega)}
\leq C\|v_\varepsilon-w_\varepsilon\|_X
\leq C\varepsilon
\end{equation}
for some constant~$C>0$.

Combining~\eqref{xdfcvgyuhnjko} and~\eqref{xdfcvgyuhnjkoxdftyujk}
we obtain that
\[
\|v-w_\varepsilon\|_{L^2(\Omega)}
\leq \|v_\varepsilon-v\|_{L^2(\Omega)}+\|v_\varepsilon-w_\varepsilon\|_{L^2(\Omega)}\leq (C+1)\varepsilon.
\]
{F}rom this we infer that~$e_k$ forms a basis of~$L^2(\Omega)$, as desired.
\end{proof}

An additional property is that any eigenvalue of problem~\eqref{eigenvalueproblem} has finite multiplicity, according to the following statement:

\begin{proposition}\label{propg}
Let~$\mu$ satisfy~\eqref{mu0}, \eqref{mu1} and~\eqref{mu2}. Take $N>2\overline s$.
%Let~$s_\sharp$ be as in~\eqref{scritico} and~$N> 2s_{\sharp}$.
Let~$\lambda_k$ be the sequence of eigenvalues given by Proposition~\ref{propde}. 

Then, each eigenvalue~$\lambda_k$ has finite multiplicity, namely, if there exists~$h\in \N$ such that
\begin{equation}\label{lalmbdakh}
\lambda_{k-1}<\lambda_k=\cdots =\lambda_{k+h}<\lambda_{k+h+1},
\end{equation}
then each eigenfunction corresponding to~$\lambda_k$ belongs to~$\mbox{span}\{e_k,\dots,e_{k+h}\}$.
\end{proposition}

\begin{proof}
Let~$h\in\N$ such that~\eqref{lalmbdakh} holds. By Proposition~\ref{propde} we infer that any element in the~$\mbox{span}\{e_k,\dots,e_{k+h}\}$ is an eigenfunction corresponding to the eigenvalue~$\lambda_k =\dots =\lambda_{k+h}$.

Thus, it only remains to show that 
\begin{equation}\label{anyeigenfunctioninspan}
\begin{split}
&\mbox{any eigenfunction~$\varphi\not\equiv 0$ corresponding to the eigenvalue~$\lambda_k =\dots =\lambda_{k+h}$}\\
&\mbox{belongs to $\mbox{span}\{e_k,\dots,e_{k+h}\}$.}
\end{split}
\end{equation}
For this, let~$S:=\mbox{span}\{e_k,\dots,e_{k+h}\}$. Moreover, let $S^\perp$ denote the orthogonal complement of $S$ in $X(\Omega)$ with respect to the scalar product introduced in \eqref{scalare+-}, namely
\[
S^\perp:=\big\{u\in X(\Omega) \;{\mbox{ s.t. }}\; \langle u, e\rangle=0 \ \mbox{ for any } e\in S\big\}.
\]
Thus, we have that~$X(\Omega)= S\oplus S^\perp$
and, since~$\varphi\in X(\Omega)$, we can write
\begin{align}\label{defvarphi12}
&\varphi:=\varphi_1 + \varphi_2 \quad \mbox{ with } \ \varphi_1\in S, \quad \varphi_2\in S^\perp \\ \label{scalarproductvarphi12}
\mbox{ and }\quad &\langle \varphi_1, \varphi_2\rangle =0.
\end{align}

We point out that if we show that
\begin{equation}\label{varphi20}
\varphi_2 \equiv 0
\end{equation}
then the claim in~\eqref{anyeigenfunctioninspan} plainly follows.
Hence, we now focus on the proof of~\eqref{varphi20}.

To start with, we test~\eqref{weakeigenvalue} for~$\varphi$ against itself and use~\eqref{defvarphi12}, \eqref{scalarproductvarphi12} and~\eqref{norma+-} to find that
\begin{equation}\label{guyhijokp}
\lambda_k \|\varphi\|^2_{L^2(\Omega)} = \|\varphi\|^2 = \|\varphi_1\|^2 + \|\varphi_2\|^2.
\end{equation}

We also claim that
\begin{equation}\label{claimvarphiL2}
\|\varphi\|^2_{L^2(\Omega)} = \|\varphi_1\|^2_{L^2(\Omega)} + \|\varphi_2\|^2_{L^2(\Omega)}.
\end{equation}
Indeed, since~$\varphi_1\in S$, we have that~$\varphi_1$ is an eigenfunction associated with the eigenvalue~$\lambda_k =\dots =\lambda_{k+h}$. Hence, by testing~\eqref{weakeigenvalue} for~$\varphi_1$ against~$\varphi_2$ and using~\eqref{scalarproductvarphi12}, we see that
\[
\lambda_k\int_\Omega \varphi_1(x) \varphi_2(x) \,dx = \langle\varphi_1, \varphi_2\rangle=0.
\]
Accordingly, since~$\lambda_k>0$ (recall~\eqref{unboundedsequence}), we get that
\[
\int_\Omega \varphi_1(x) \varphi_2(x) \,dx=0.
\]
{F}rom this and~\eqref{defvarphi12} we infer that
\[
\|\varphi\|^2_{L^2(\Omega)} = \|\varphi_1 +\varphi_2\|^2_{L^2(\Omega)} = \|\varphi_1\|^2_{L^2(\Omega)} + \|\varphi_2\|^2_{L^2(\Omega)},
\]
which is the desired claim in~\eqref{claimvarphiL2}.

Now, by~\eqref{defvarphi12}, we have that for any~$j\in\{k, \dots, k+h\}$ there exists~$c_j\in\R$ such that
\[
\varphi_1 (x) = \sum_{j=k}^{k+h} c_j e_j(x).
\]
{F}rom this, \eqref{lalmbdakh}, \eqref{norma+-}
and Proposition~\ref{propf}, we obtain that
\begin{equation}\label{tguyhijokp}
\begin{split}&
\|\varphi_1\|^2 = \sum_{j=k}^{k+h} c^2_j \,\|e_j\|^2 = \sum_{j=k}^{k+h} c^2_j \lambda_j =\lambda_k \sum_{j=k}^{k+h} c^2_j\\
&\qquad = \lambda_k \sum_{j=k}^{k+h} c^2_j
\|e_j\|^2_{L^2(\Omega)}
= \lambda_k \|\varphi_1\|^2_{L^2(\Omega)}.
\end{split}
\end{equation}

Moreover, we observe that~$\varphi_2$ is an eigenfunction corresponding to the eigenvalue~$\lambda_k$. Indeed, by~\eqref{guyhijokp}, \eqref{claimvarphiL2} and~\eqref{tguyhijokp} we get
\begin{eqnarray*}
0&=& \|\varphi_1\|^2 - \lambda_k \|\varphi_1\|^2_{L^2(\Omega)}\\&
=&\lambda_k \|\varphi\|^2_{L^2(\Omega)} - \|\varphi_2\|^2-\lambda_k \|\varphi_1\|^2_{L^2(\Omega)}\\&
=& \lambda_k \|\varphi_2\|^2_{L^2(\Omega)} - \|\varphi_2\|^2,
\end{eqnarray*}
as desired.

{F}rom this fact, \eqref{lalmbdakh} and Lemma~\ref{lemmino2}, it follows that
\[
\langle \varphi_2, e_j\rangle = 0 \quad\mbox{ for any } j\in\{1, \dots, k-1\}.
\]
Thus, combining this and~\eqref{scalarproductvarphi12}, and recalling~\eqref{Ek+1definition}, we deduce that~$\varphi_2 \in E_{k+h+1}$.

Now, we suppose by contradiction that the claim in~\eqref{varphi20} do not hold, that is~$\varphi_2\not\equiv 0$. Therefore, by~\eqref{kautovalore} and~\eqref{lalmbdakh}, and since~$\varphi_2 \in E_{k+h+1}$, we see that
\begin{eqnarray*}&&
\lambda_k <\lambda_{k+h+1} = \min_{u\in E_{k+h+1}\setminus \{0\}} \dfrac{\displaystyle\int_{[0, 1]} [u]^2_s \, d\mu^+(s)-\int_{[0, \overline s)} [u]^2_s \, d\mu^-(s)}{\displaystyle\int_\Omega |u(x)|^2 \,dx}\\&&\qquad
\le \ \dfrac{\displaystyle\int_{[0, 1]} [\varphi_2]^2_s \, d\mu^+(s)-\int_{[0, \overline s)} [\varphi_2]^2_s \, d\mu^-(s)}{\displaystyle\int_\Omega |\varphi_2(x)|^2 \,dx},
\end{eqnarray*}
that is
\begin{equation}\label{tyguhijnkm}
\|\varphi_2\|^2 > \lambda_k \|\varphi_2\|^2_{L^2(\Omega)}.
\end{equation}
Thus, by~\eqref{guyhijokp}, \eqref{claimvarphiL2}, \eqref{tguyhijokp} and~\eqref{tyguhijnkm}, we infer that
\[
\begin{split}
\lambda_k \|\varphi\|^2_{L^2(\Omega)} &= \|\varphi_1\|^2 + \|\varphi_2\|^2\\
& >\lambda_k \left(\|\varphi_1\|^2_{L^2(\Omega)} + \|\varphi_2\|^2_{L^2(\Omega)}\right)\\
&= \lambda_k \|\varphi\|^2_{L^2(\Omega)},
\end{split}
\]
which is a contradiction.  Hence, \eqref{varphi20} holds and then the desired claim in~\eqref{anyeigenfunctioninspan} is proved.
\end{proof}

We are now ready to prove Theorem~\ref{thmautovalori}:

\begin{proof}[Proof of Theorem~\ref{thmautovalori}]
The desired results plainly follow combining Propositions~\ref{propde}, \ref{propf} and~\ref{propg}.
\end{proof}

\begin{appendix}

\section{Proof of~\eqref{counterexamplemaxprinciple}}\label{appendixmaxprinciple}

In this appendix we provide an explicit counterexample to the maximum principle
which establishes~\eqref{counterexamplemaxprinciple}.

Let~$\Omega:=(-1, 1)$, $s\in (0, 1)$, $\alpha>0$, $R>0$ and~$u_R:\R\to\R$ be the function defined as
\[
u_R(x):=\begin{cases}
x^2-1 &\mbox{ for } |x|\le 1+R,\\
(1+R)^2-1 &\mbox{ for } |x|> 1+R.
\end{cases}
\]
Notice that~$u_R\in C(\R)\cap C^2(\Omega)$ and that~$u_R\ge0$ in~$\R\setminus(-1,1)$ and~$u_R<0$ in~$(-1,1)$.

We now check that, if~$R$ is sufficiently large,
\begin{equation}\label{claimprincipale}
-\Delta u_R - \alpha(-\Delta)^s u_R \ge 0\quad\mbox{ in }\Omega .
\end{equation}
We observe that, since~$-\Delta u_R = -2$ in~$\Omega$,  to obtain the desired result~\eqref{claimprincipale} we only need to prove that,
if~$R$ is sufficiently large,
\begin{equation}\label{claimfrazionario}
- (-\Delta)^s u_R \ge \frac{2}{\alpha} \quad\mbox{ in } \Omega .
\end{equation}

We also notice that, in this setting, the space~$C_\mu(\R)$ in~\eqref{inteGRAeiuwot2} boils down to
$$ \left\{u\in C(\R) \mbox{ such that }\int_{\R}\frac{|u(x)|}{1+|x|^{1+2s}}\,dx<+\infty \right\}$$
and the function~$u_R$ belongs to this space.

Hence, in light of Proposition~\ref{propcompute}, we can compute~$(-\Delta)^s u_R(x)$ pointwise. Therefore, we write
\begin{equation}\label{I1I2I3}
\begin{split}
&(-\Delta)^s u_R(x) =I_1 +I_2 +I_3,\\{\mbox{where }} \qquad&
I_1:=\int_{-\infty}^{-1-R} \frac{x^2 -(1+R)^2}{|x-y|^{1+2s}}\, dy,\\&
I_2:=\int_{1+R}^{+\infty} \frac{x^2 -(1+R)^2}{|x-y|^{1+2s}} \,dy \\
{\mbox{and }}\qquad &I_3:= PV \int_{-1-R}^{1+R} \frac{x^2 -y^2}{|x-y|^{1+2s}}\, dy.
\end{split}
\end{equation}
We compute separately~$ I_1$, $I_2$ and~$I_3$. We have
\begin{equation}\label{A3BIS}
I_1 = (x^2 -(1+R)^2) \int_{-\infty}^{-1-R} \frac{dy}{(x-y)^{1+2s}} = -\frac{(1+R-x)(1+R+x)^{1-2s}}{2s}
\end{equation}
and
\begin{equation}\label{A3TER}
I_2 = (x^2 -(1+R)^2) \int_{1+R}^{+\infty} \frac{dy}{(y-x)^{1+2s}} = -\frac{(1+R+x)(1+R-x)^{1-2s}}{2s}.
\end{equation}

Moreover, for~$\varepsilon\in(0,1)$ sufficiently small,
\begin{eqnarray*}
&&\int_{-1-R}^{x-\varepsilon} \frac{x+y}{(x-y)^{2s}} \,dy - \int_{x+\varepsilon}^{1+R} \frac{x+y}{(y-x)^{2s}}\, dy\\&=&
x\left(\;\int_{-1-R}^{x-\varepsilon} \frac{dy}{(x-y)^{2s}} - \int_{x+\varepsilon}^{1+R} \frac{dy}{(y-x)^{2s}}\right)
+\int_{-1-R}^{x-\varepsilon} \frac{y}{(x-y)^{2s}}\, dy - \int_{x+\varepsilon}^{1+R} \frac{y}{(y-x)^{2s}} \,dy\\&=&
2x\left(\;\int_{-1-R}^{x-\varepsilon} \frac{dy}{(x-y)^{2s}} - \int_{x+\varepsilon}^{1+R} \frac{dy}{(y-x)^{2s}}\right)
-\int_{-1-R}^{x-\varepsilon} (x-y)^{1-2s}\, dy - \int_{x+\varepsilon}^{1+R}(y-x)^{1-2s} \,dy.
\end{eqnarray*}
As a result, if~$s\neq\frac12$ we have that
\begin{equation}\label{Arquattro}\begin{split}
I_3&= \lim_{\varepsilon\searrow 0}\left(\;
\int_{-1-R}^{x-\varepsilon} \frac{x+y}{(x-y)^{2s}} \,dy - \int_{x+\varepsilon}^{1+R} \frac{x+y}{(y-x)^{2s}}\, dy\right)\\&=  \lim_{\varepsilon\searrow 0}\left[
\frac{2x}{1-2s}\Big( (1+R+x)^{1-2s}- (1+R-x)^{1-2s} \Big)\right.
\\&\qquad\left.-\frac{1}{2-2s}\Big( (1+R+x)^{2-2s}+ (1+R-x)^{2-2s}-2\varepsilon^{2-2s}
\Big)\right]\\&=
\frac{2x}{1-2s}\Big( (1+R+x)^{1-2s}- (1+R-x)^{1-2s} \Big)\\&\qquad-\frac{1}{2-2s}\Big( (1+R+x)^{2-2s}+ (1+R-x)^{2-2s} \Big),
\end{split}\end{equation}
while if~$s=\frac12$ we see that
\begin{equation}\label{7382trgsakfkjerutrjke}\begin{split}
I_3&= \lim_{\varepsilon\searrow 0}\left(\;
\int_{-1-R}^{x-\varepsilon} \frac{x+y}{x-y} \,dy - \int_{x+\varepsilon}^{1+R} \frac{x+y}{y-x}\, dy\right)\\&=  \lim_{\varepsilon\searrow 0}\left[
2x\log\left(\frac{1+R+x}{1+R-x}\right) -2-2R+2\varepsilon
\right]\\&=2x\log\left(\frac{1+R+x}{1+R-x}\right) -2-2R.
\end{split}\end{equation}

Thus, by~\eqref{I1I2I3}, \eqref{A3BIS}, \eqref{A3TER} and~\eqref{Arquattro}, we obtain that, if~$s\neq\frac12$, for any~$x\in \Omega$,
\[
\begin{split}
(-\Delta)^s u_R(x)&=-\frac1{2s}\Big( (1+R-x)(1+R+x)^{1-2s}+(1+R+x)(1+R-x)^{1-2s}\Big)\\&\qquad +
\frac{2x}{1-2s}\Big( (1+R+x)^{1-2s}- (1+R-x)^{1-2s} \Big) \\&\qquad-\frac{1}{2-2s}\Big( (1+R+x)^{2-2s}+ (1+R-x)^{2-2s} \Big)\\
&=:g_{R,s}(x).
\end{split} 
\]
We observe that~$g_{R,s}$ is even and has a maximum at~$x=1$ in~$[0,1]$.
Therefore, in order to prove~\eqref{claimfrazionario}, we only need
to show that
\begin{equation}\label{claimgs}
g_{R,s}(1)\leq -\frac{2}{\alpha}.
\end{equation}
To this end, we write
\begin{eqnarray*}
g_{R,s}(1)&=&-\frac1{2s}\Big( R(2+R)^{1-2s}+(2+R)R^{1-2s}\Big)\\&&\qquad +
\frac{2}{1-2s}\Big( (2+R)^{1-2s}- R^{1-2s} \Big)  -\frac{1}{2-2s}\Big( (2+R)^{2-2s}+ R^{2-2s} \Big)
\\&=&
R(2+R)^{1-2s}\left(-\frac{1}{2s}+\frac{2}{(1-2s)R}-\frac{1}{(1-s)R}-\frac{1}{2-2s}\right)\\&&\qquad
+R^{2-2s}\left(-\frac{1}{sR}-\frac{1}{2s}-\frac{2}{(1-2s)R}-\frac{1}{2-2s}\right).
\end{eqnarray*}
We point out that
\[
\lim_{R\to +\infty}g_{R,s}(1)=-\infty.
\]
{F}rom this, we infer that, for any $\alpha>0$, there exists $R_0>0$ such that
\[
g_{R_0,s}(1)\leq -\frac{2}{\alpha}.
\]
%A standard analysis shows that, for any~$s\in(0,1)$ with~$s\neq 1/2$, 
%\[
%\frac{3^{1-2s}(8s-2)+16s^3-24s^2+16s-6}{2s(1-2s)(2-2s)}\leq 0.
%\]
%which is equivalent to~\eqref{claimgs}.
This proves the claim in~\eqref{claimfrazionario}, and thus~\eqref{claimprincipale} holds in the case~$s\neq 1/2$.

When~$s=1/2$, by~\eqref{I1I2I3}, \eqref{A3BIS}, \eqref{A3TER} and~\eqref{7382trgsakfkjerutrjke}, we find that
\[
(-\Delta)^{\frac12} u_R(x)=2x \log\left(\frac{1+R+x}{1+R-x}\right) -4-4R.
\]
We observe that the function
\[
h_R(x):=2x \log\left(\frac{1+R+x}{1+R-x}\right) -4-4R
\]
is even and has a maximum at~$x=1$ in~$[0,1]$. Moreover, 
\[
\lim_{R\to +\infty}h_R(1) =\lim_{R\to +\infty}2 \log\left(\frac{2+R}{R}\right) -4-4R=-\infty.
\]
Thus, for any $\alpha>0$ there exists $R_0>0$ such that
\[
h_{R_0}(1)\leq -\frac{2}{\alpha}.
\]
This proves the claim in~\eqref{claimfrazionario} and then~\eqref{claimprincipale} is established also in the case~$s=1/2$. 

The proof of~\eqref{counterexamplemaxprinciple} is thereby complete.

\section{Proof of~\eqref{baseX}}\label{appendixFourier}
In this appendix we prove that the sequence of eigenfunctions~$e_k$ defined in Proposition~\ref{propde} is a basis for the space~$X(\Omega)$.

To this aim, let~$f\in X(\Omega)$. For any~$i$, $j\in \N$, we define
\[
\widetilde{e_i}:=\frac{e_i}{\|e_i\|}
\qquad \mbox{and} \qquad
f_j:=\sum_{i=1}^j\widetilde{e_i}\,\langle f, \widetilde{e_i}\rangle.
\]
Observe that, for any~$j\in \N$,
\begin{equation}\label{xdrtgvbhuijn}
f_j\in \mbox{span}\{e_1,\dots, e_j\}.
\end{equation}
Let~$v_j:=f-f_j$. Then, by the orthonormality of~$\widetilde{e_k}$ in~$X(\Omega)$ with respect to the norm defined in~\eqref{norma+-} (see Proposition~\ref{propnormeequiv}), we get that
\[
0\leq \|v_j\|^2=\|f\|^2+\|f_j\|^2-2\langle f, f_j\rangle=\|f\|^2 -\sum_{i=1}^j\langle f, \widetilde{e_i}\rangle^2.
\]
Thus, for any~$j\in \N$,
\[
\sum_{i=1}^j\langle f, \widetilde{e_i}\rangle^2 \leq \|f\|^2.
\]
This implies that 
\[
\sum_{i=1}^{+\infty}\langle f, \widetilde{e_i}\rangle^2 \mbox{ is a convergent series}.
\]
Hence, we define 
\[
\tau_j:=\sum_{i=1}^j\langle f, \widetilde{e_i}\rangle^2
\]
and we have that
\begin{equation}\label{zsedxcftgvbhju}
\tau_j \mbox{ is a Cauchy sequence in }\R.
\end{equation}

Moreover, again by the orthonormality of~$\widetilde{e_k}$ in~$X(\Omega)$, if~$l>j$, we see that
\[
\|v_l-v_j\|^2=\left\|\sum_{i=j+1}^l\widetilde{e_i} \langle f, \widetilde{e_i}\rangle\right\|^2=\sum_{i=j+1}^l\langle f, \widetilde{e_i}\rangle^2=\tau_l-\tau_j.
\]
This and~\eqref{zsedxcftgvbhju} imply that~$v_j$ is a Cauchy sequence in~$X(\Omega)$. By~\eqref{XHilbert},
\begin{equation}\label{zcvgtgbhuimkl}
\mbox{$v_j$ converges strongly to some $v$ in $X(\Omega)$.}
\end{equation}

We observe that, for any~$j\geq k$,
\[
\langle v_j, \widetilde{e_k}\rangle
=\langle f, \widetilde{e_k}\rangle-\langle f_j, \widetilde{e_k}\rangle
=\langle f, \widetilde{e_k}\rangle-\langle f, \widetilde{e_k}\rangle=0.
\]
Taking the limit as~$j$ goes to infinity, by~\eqref{zcvgtgbhuimkl}
we infer that~$\langle v, \widetilde{e_k}\rangle=0$ for any~$k \in \N$. Thus,
the statement in~\eqref{claimbaseX} implies that~$v\equiv 0$.

Recalling that~$f_j=f-v_j$, we thereby obtain that
\[
\mbox{$f_j$ converges strongly to $f$ in $X(\Omega)$.}
\]
This, together with~\eqref{xdrtgvbhuijn}, gives~\eqref{baseX},
as desired.

\end{appendix}

\section*{Acknowledgements} 
SD, CS and EV are members of the Australian Mathematical Society (AustMS). CS and EPL are members of the INdAM--GNAMPA.

CS acknowledges the support of the Juan de la Cierva Fellowship (grant number JDC2023-
050365-I).

This work has been supported by the Australian Laureate Fellowship FL190100081
and by the Australian Future Fellowship
FT230100333.

\vfill

\end{document}